\documentclass[11pt]{article}
\usepackage{amsmath,amstext,amssymb,amsthm}

\usepackage[english]{babel}
\usepackage[T1]{fontenc}
\usepackage[utf8]{inputenc}
\usepackage{tikz}
\usepackage{url}
\usepackage{times}

\usepackage{srcltx} 
\usepackage{enumerate}
\usepackage[utf8]{inputenc}
\usepackage{mathrsfs}
\usepackage{xspace}
\usepackage{framed}
\usepackage{enumerate}
\usepackage{xcolor}
\usepackage{xspace}
\usepackage{footnote}
\makesavenoteenv{tabular}
\usepackage[colorlinks=true,citecolor=black,linkcolor=black,urlcolor=blue]{hyperref}
\renewcommand{\le}{\leqslant}
\renewcommand{\ge}{\geqslant}

\newtheorem{theorem}{Theorem}
\newtheorem{lemma}[theorem]{Lemma}

\newtheorem{corollary}[theorem]{Corollary}
\newtheorem{remark}[theorem]{Remark}

\theoremstyle{definition}
\newtheorem{example}{Example}
\newtheorem{problem}{Problem}

\def\ceil#1{\left\lceil #1 \right\rceil}

\newcommand{\Hall}{\mathbf{m}}
\newcommand{\Al}{\Sigma}
\newcommand{\Ex}{\mathit{D}}
\newcommand{\ex}{\mathit{d}}

\begin{document}
\title{Avoiding square-free words on free groups}

\author{Golnaz Badkobeh\footnote{Department of Computing, Goldsmiths University of London, United Kingdom. \texttt{G.Badkobeh@gold.ac.uk}}
\and Tero Harju\footnote{Department of Mathematics and Statistics, University of~Turku, Finland. \texttt{harju@utu.fi}}
\and Pascal Ochem\footnote{LIRMM, CNRS, Universit\'e de Montpellier, France. \texttt{ochem@lirmm.fr}}
\and Matthieu Rosenfeld\footnote{LIRMM, CNRS, Universit\'e de Montpellier, France. \texttt{rosenfeld@lirmm.fr}}
}

\maketitle

\begin{abstract}
We consider sets of factors that can be avoided in square-free words on two-generator free groups.
The elements of the group are presented in terms of $\{0,1,2,3\}$ such that $0$ and $2$ 
(resp.,$1$ and $3$) are inverses of each other
so that $02$, $20$, $13$ and $31$ do not occur in a reduced word. 
A Dean word is a reduced word that does not contain occurrences of $uu$ for any nonempty~$u$.
Dean showed in 1965 that there exist infinite square-free reduced words.
We show that if $w$ is a Dean word of length at least~59
then there are at most six reduced words of length~3 avoided by $w$. 
We construct an infinite Dean word avoiding six reduced words of length~3.
We also construct infinite Dean words with low critical exponent and avoiding fewer reduced words of length~3.
Finally, we show that the minimal frequency of a letter in a Dean word is $\frac{8}{59}$
and the growth rate is close to 1.45818.
\end{abstract}

\section{Introduction}

Axel Thue~\cite{Thue:12} showed in 1912 that there exists an infinite square-free word
$\Hall$ over the alphabet $\{0,1,2\}$ that avoids occurrences of the words 
$010$ and $212$ of length~3. We shall consider a related question on the
free group generated by two elements.

For the basic notions in combinatorics on words, we refer to Lothaire~\cite{Lothaire, Lothaire2}.
Let
\[
\Al_k=\{0,1, \ldots, k-1\}
\]
be an alphabet of $k$ letters. We denote by $\Al^*$ 
the set of all finite words over an alphabet $\Al$.
We are interested solely in the words of $\Al_3^*$ and $\Al_4^*$.
The set of infinite words $w\colon \mathbb{N} \to \Al$ over 
an alphabet $\Al$, 
represented here as infinite strings $w(1)w(2) \cdots$, is denoted
by $\Al^\omega$. 

The length, i.e., the number of occurrences of letters, of a word $w \in \Al^*$ is denoted by $|w|$. 
 
A factor $v$ of $w$ is \emph{right-special} in $w$ if there exist at least two distinct letters $a$ and $b$ such that $va$ and $vb$ both are factors of $w$.

A word $w \in \Al^* \cup \Al^\omega$ is said to \emph{avoid}
another word $v$, if $v$ is not a \emph{factor} (i.e., a finite contiguous subword) of $w$. 
Furthermore, $w$ is \emph{square-free} if it avoids  all nonempty words of the form $vv$. 
A morphism, i.e., a substitution of letters to words,
$h\colon \Al^* \to \Delta^*$ is said to be \emph{square-free}, if it preserves square-freeness,
i.e., if $h(w)$ is square-free for all square-free words $w$. 

\begin{example}
The word $w=01210120210121021202101202120$ of length 29 is, up to permutations of the letters,
the longest ternary square-free word $w \in \Al_3^*$ avoiding three square-free words of length~3.
It avoids the words $010,020,201$.
There are infinite square-free ternary words that avoid two words of length~3; the Hall-Thue word
$\Hall$ as defined below is an example of these. 
\qed\end{example}

The following result is due to Crochemore~\cite{Crochemore}.

\begin{theorem}\label{Crochemore}
A morphism $\alpha\colon \Al_3^* \to \Delta^*$ is square-free if and only if it preserves all square-free words of length 5.
\end{theorem}

The simplest square-free morphism is due to Thue~\cite{Thue:06} ; see Lothaire~\cite{Lothaire},
\[
h_{T}(0)=01201\,, \
h_{T}(1)=020121\,, \
h_{T}(2)=0212021\,.
\]

The \emph{Hall-Thue} word $\Hall=\tau^\omega(0)$, 
also known as a \emph{variation of Thue-Morse word}~\cite{Blanchet-Sadri:integers},
is obtained by iterating the (Hall) morphism~\cite{Hall}:
\[
\tau(0)=012\,, \
\tau(1)=02\,, \
\tau(2)=1\,.
\]
Thus 
$
\Hall=012 02 1 012 1 \cdots
$ 
The morphism $\tau$ is not square-free since $\tau(010)$ contains a square $2020$.
However, the Hall-Thue word $\Hall$ is square-free as was shown by Thue~\cite{Thue:06} in 1906. 
The word $\Hall$ avoids the words $010$ and $212$ as is immediate from the form of $\tau$.

\section{Dean words}

While considering \emph{reduced words} of a free group on two generators,
we work on the alphabet $\Al_4$ by considering $0$ and $2$ (resp.\ $1$ and $3$)
as inverses of each other. Thus a reduced word $w \in \Al_4^*$ does not have factors from the set 
$\{02,20,13,31\}$. In other words, the even and odd letters alternate in the reduced words.

A reduced word $w \in \Al^*_4$ is said to be a \emph{Dean word} if it is square-free. Dean words are moderately numerous as suggested in Table~\ref{DeanNumber}.
 
 \begin{table}[htp]
\begin{center}
\begin{tabular}{|c|c||c|c||c|c||c|c||c|c|}\hline
n&\#&n&\#&n&\#&n&\# &n&\#\cr\hline
1&4&2&8&3&16&4&24&5&40\cr\hline
6&64&7&104&8&144&9&216&10&328\cr\hline
11&496&12&720&13&1072&14&1584&15&2344\cr\hline
16&3384&17&4952&18&7264&19&10632&20&15504\cr\hline
\end{tabular}
\end{center}
\caption{The number of Dean words, i.e., square-free reduced words over $\Al_4$.}\label{DeanNumber}
\end{table}

Dean~\cite{Dean} proved in 1965 the existence of an infinite square-free reduced word.

\begin{theorem}\label{Dean}
There exists an infinite Dean word.
\end{theorem}

We give here three simple proofs.

\begin{proof}[First Proof]
As noticed by Baker, McNulty and Taylor~\cite{BNT89}, the construction of Dean~\cite{Dean}
corresponds to the fixed point 
\[
f^\omega(0)=01210321012303210121032301230321\cdots
\]
of the simple morphism $f\colon \Al_4^* \to \Al_4^*$ defined by
\begin{equation}\label{Shallit}
f(0)= 01\,, \ f(1) = 21\,, \ f(2) = 03\,, \ f(3) = 23\,.
\end{equation}

Clearly, the iterated word $f^\omega(0)$ is reduced.
Its square-freeness can also be checked with the Walnut software~\cite{Walnut}, as communicated to us by Jeffrey Shallit~\cite{Shallit}:
\begin{verbatim}
eval dean1 "Ei,n (n>=1) & At (t<n) => DE[i+t]=DE[i+n+t]"
\end{verbatim}
Interestingly, $f^\omega(0)$ is a complete shuffle of $(02)^\omega$ and the paper folding sequence
$1131133111331331\cdots$ (after renaming $0$ to $3$); see Davis and Knuth~\cite{Davis-Knuth}
or Allouche and Shallit~\cite{Allouche-Shallit}.
The latter claim was also verified by Shallit using Walnut, or by
considering the substitution rules for the paper folding sequence:
\[
 11 \mapsto 1101\,, \ 01 \mapsto 1001\,, \ 10 \mapsto 1100\,, \ 00 \mapsto 1000\,.
\]

We can modify \eqref{Shallit} by combining the
values $f(1)$ and $f(2)$ to obtain $h\colon \Al_3^* \to \Al_4^*$ from the three letter alphabet:
\[
h(0)=01\,, \ h(1)=23\,, \ h(2)=2103\,.
\]	
The morphism $h$ is not square-free, since $h(212)=210(3232)103$ contains a square.
However, we can check by Theorem~\ref{Currie-thue} that $h(\Hall)$ is square-free.
\end{proof}

\begin{proof}[Second Proof]
Consider the morphism $h\colon \Al_3^* \to \Al_4^*$ given by the images of the letters,
\[
h(0)=10\,, \ 
h(1)=32\,, \ 
h(2)=1230103012\,.
\]
It is not difficult to show directly that $h$ is square-free, but Theorem~\ref{Crochemore} does it also quickly.
Since the words $h(ab)$, for different $a, b \in \Al_3$, are reduced, the infinite word $h(\Hall)$ is a Dean word. 
\end{proof}

\begin{proof}[Third Proof]
Let $w \in \Al_3^{\omega}$ be an infinite (ternary) square-free word. 
We construct a word $\overline{w} \in \Al_{4}^{\omega}$ by 
adding the letter $3$ in the middle of every occurrence of $02$ and $20$. Then also
$\overline{w}$ is square-free. Indeed, if there were a square 
$\overline{u}\overline{u}$ in $\overline{w}$,
the occurrences of $3$ would be aligned (in the same positions) in the two instances of $\overline{u}$, and
deleting the occurrences of $3$ would evoke a square $uu$ into $w$; a contradiction.
Clearly, $\overline{w}$ is reduced by its construction.
\end{proof}

A morphism $h\colon \Al_3^* \to \Al_4^*$ that preserves Dean words, is called a \emph{Dean morphism}.
For $h$ to be a Dean morphism, it suffices to check that $h(w)$ is square-free
for all square-free $w$ of length~5 (Theorem~\ref{Crochemore})
and checking that the words $h(ab)$, for different $a, b \in \Al_3$, are reduced.

One can apply the third proof of Theorem~\ref{Dean} to finite words.
Then applying this operation to the Thue morphism $h_T$, the images of which start with the letter 0,
and end with $1$, yields the following Dean morphism
\[
\overline{h}_T(0)=012301\,, \
\overline{h}_T(1)=03230121\,, \
\overline{h}_T(2)=0321230321\,.
\]
The reverse operation does not work; deleting the occurrences of 3 from a Dean word 
does not necessarily give a square-free ternary word.
\section{Sets of avoided words}
An avoided word is also called an \emph{absent} word. Moreover, $v$ is a \emph{minimal} absent word of $w$ if it is absent in $w$ and all its proper prefixes and suffixes occur in $w$.

For a Dean word $w \in \Al_4^* \cup \Al_4^\omega$, let $\Ex_l(w)$ denote the set of minimal absent words of length $l$ of $w$ and $\ex_l(w) = |\Ex_l(w)| \quad (\text{the size of } \Ex_l(w))$.

We are interested in the avoided (absent) words of length~3.

With the aid of a computer program, we find that the
longest Dean word that avoids the word $w=10$ of length~2 has length~58:
\[
3230123032301232123032301230321230123212303230123212301232 
\]

By considering the permutations of $\Al_4$ that preserve reduced words, we have

\begin{lemma}\label{max58}
Let $w$ be a Dean word with $|w| \ge 59$. Then $\ex_2(w)=0$.
\end{lemma}

We have already seen, by the third proof of Theorem~\ref{Dean}, that every
ternary square-free word $w\in \Al_3^*$ gives rise to a Dean word $\overline{w} \in \Al_4^*$.
There are 15504 Dean words of length~20 divided into 709 different sets of avoided words including
the empty set. Though up to the eight permutations of the letters
that preserve reduced words (see the Appendix), there are `just' 1938
Dean words of length~20.

The word $w=01030121012321230323$ of length~20 is an example of a Dean word that 
accommodates all reduced words of length~3, i.e., $\ex_3(w) = 0$. The length~20 is optimal:
a Dean word of length~19 always avoids some word of length~3.
Note that $w$ is not `de Bruijn -type' since both $012$ and $123$ occur twice in $w$. 

\begin{example}
The Hall-Thue word $\Hall$ avoids $010$ and $212$.
The corresponding Dean word $\overline{\Hall}$ has 
$\Ex_3(\overline{\Hall})=\{010, 030, 212, 232\}$ simply since the letter $3$ is never added between two $0$'s or two $2$'s.
\qed\end{example}

The following two simple lemmas deal with \emph{repetitions} of the form $abvab$ for a pair $ab$.
They are verified by the aid of a computer.
Up to permutations, the words $010321230$ and $012303210$ are the maximum length Dean words without repetitions of pairs.

\begin{lemma}\label{norep}
A Dean word that has no repetitions of pairs, has length at most 9. 
\end{lemma}

A reduced word, $w$ is \emph{extendable} if there is at least a letter $a \in \Al_4$, such that $wa$ is a reduced word, $wa$ ia called the extension of $w$.
Clearly, any reduced word has two extentions.

\begin{lemma}\label{norep:2}
Let $w$ be a Dean word that has only one right-special factor of length 2.
Then $|w|\le 23$. 
\end{lemma}

\begin{example}\label{ex:Palind}
The word $w=01032301032123010323010$ is a longest Dean word with only
one right-special factor of length 2, $v=32$.
We have $\ex_2(w)=0$ and $\Ex_3(w)=\{012,030,101,121,210,232,303\}$ of seven elements.
\qed\end{example}

A prefix $v$ of $w$ is \emph{proper} if $|v| < |w|$.

\begin{lemma}\label{lemm:24}
Let $w$ be a Dean word of length $|w| \ge 24$ such that every reduced pair occurs in a proper prefix of $w$.
Then $\ex_3(w) \le 6$.
\end{lemma}

\begin{proof}
By the hypothesis, each of the eight reduced pairs $ab$ has at least one extension in $w$. 
Hence $\ex_3(w) \le 8$.

Suppose that $\ex_3(w)= 8$. Then no reduced pair $ab$ is a right-special factor in~$w$. 
However, by Lemma~\ref{norep}, each Dean word of length~10 has a repetition of some pair $abzab$. 
By the uniqueness of the extensions, this would eventually result in a square
$abz{\cdot}abz$ in the prefix of $w$ of length~16.

The case $\ex_3(w)= 7$ allows that exactly one pair is a right-special factor in $w$, and
the other seven pairs are not right-special.
By Lemma~\ref{norep:2}, $|w|\le 23$, and the claim follows.
\end{proof}

The bound 24 in Lemma~\ref{lemm:24} is optimal as witnessed by Example~\ref{ex:Palind}.

\begin{corollary}\label{inf:6}
If $w$ is an infinite Dean word then $\ex_3(w) \le 6$.
\end{corollary}

By applying a computer search on all Dean words of length~59, we see that there are only four possible sets
of avoided Dean words, as specified in Theorem~\ref{four}.
The claim of Theorem~\ref{four} for words $|w| >59$ immediately follows from this result.
We shall show in Theorem~\ref{Actual3} that these sets are actual for infinite Dean words.

\begin{theorem}\label{four}
Let $w$ be a Dean word of length $|w| \ge 59$ with $\ex_3(w) = 6$.
Then $\Ex_3(w)$ is one of the following sets
\begin{align*}
S_1&=\{101,123,212,232,303,321\}\\
S_2&=\{012,030,101,121,210,232\}\\
S_3&=\{010,032,212,230,303,323\}\\
S_4&=\{010,030,103,121,301,323\}
\end{align*}
\end{theorem}

As expected, the family of the sets $S_i$ in Theorem~\ref{four} is closed
under permutations that preserve reduced words.
E.g.,
the third set $S_3$ can be obtained from $S_1$ by
the permutation $(0\, 1)(2\, 3)$, and $(1\, 3)$ fixes $S_1$.
Moreover, each $S_i$ is closed under reversals of their elements.

\begin{remark}
For the length~58, there are still 12 different sets of cardinality six of avoided words. 
The drop in the number between the lengths~58 to~59 is dramatic
due to the fact that there are no longer any avoided words of length~2.
\end{remark}

The next proof relies on a strong result due to Currie~\cite{Currie:Thue}.

\begin{theorem}\label{Currie-thue}
Let $h \colon \Al_3^* \to \Delta^*$ be a morphism,
and let $w \in \Al_3^\omega$ be an infinite word with 
$\{010, 212\} \subseteq \Ex_3(w)$.
Then $h(w)$ is square-free if and only if $h$ preserves square-freeness of 
the factors of $w$ of length~7.
\end{theorem}

We apply Theorem~\ref{Currie-thue} to the Hall-Thue word $\Hall$
which does avoid the words $010$ and $212$.

\begin{corollary}\label{cor:Currie}
Let $h \colon \Al_3^* \to \Al_4^*$ be a non-erasing morphism.
Then $h(\Hall)$ is square-free if and only if $h$ is square-free on the factors of $\Hall$ of length~7.
\end{corollary}

For the sake of completeness, the factors of length~7 of $\Hall$ are listed in the Appendix.

\begin{theorem}\label{Actual3}
There exists an infinite Dean word $w \in \Al_4^\omega$ such that  $\Ex_3(w)=S_i$, for each $i=1,2,3,4$.
\end{theorem}

\begin{proof}
We obtain a solution $h(\Hall)$ for the first set
\[
S_1=\{101,123,212,232,303,321\};
\] 
solutions for the other three sets are obtained by suitable permutations
applied to the following morphism.
 
Let $g\colon \Al_3^* \to \Al_4^*$ be defined by
\[
g(0)=103230\,, \
g(1)=1030\,, \
g(2)=12\,,
\]	
where the images have lengths $6,4,2$, respectively.
The morphism $g$ is \emph{not} square-free since the image $g(010)$ contains a square.
However, a computer check can verify that $g$ preserves square-freeness of the factors of $\Hall$ of length~7.
Therefore, by Corollary~\ref{cor:Currie}, the word $g(\Hall)$ is square-free.
Clearly, the images $g(w)$ are reduced for all $w \in \Al_3^*$.
By inspection, the words of $S_1$ are avoided by all words $g(ab)$ with $a \ne b$. Since $\ex_3(g(\Hall)) \le 6$ by Corollary~\ref{inf:6}, 
it follows that $\Ex_3(g(\Hall))=S_1$.
\end{proof}

By the next theorem and the remark that follows, $g$ is, in effect, the only morphism 
$\Al_3^* \to \Al_4^*$ that produces Dean words avoiding the words of $S_1$.

\begin{theorem}\label{unique}
Let $w$ be an infinite Dean word such that $S_1 \subseteq \Ex_3(w)$.
Then $w = u g(v)$, where $|u|\le5$, and $v$ is a square-free ternary word.
\end{theorem}

\begin{proof}
Let $w=x1z$, where $x$ has no occurrences of $1$. (The word $x$ can be empty.)
By analysing the words without factors in $S_1$, we conclude that
$1$ is preceded in $w$ only by a suffix of $g(0)=103230$, $g(1)=1030$, or $g(2)=12$.
This gives the bound $|x|\le|g(0)|-1=5$ for the prefix~$x$.
We conclude that the suffix $1z$ has a decomposition in terms of $g(a)$, $a \in \Al_3$.
This means that $w=ug(v)$, where $v$ is square-free since $w$ is.
\end{proof}

\begin{remark}
The set $S_1$ is closed under the permutation $(1\, 3)$, and thus the morphism
$g'$, where $1$ and $3$ are interchanged, also satisfies Theorem~\ref{unique}.
Also, besides permuting with $(1\, 3)$, we can conjugate the images with respect to a common prefix: 
if $g(a)=uv_a$, for all $a \in \Al_3$ then also $g''$ with $g''(a)=v_au$
satisfies Theorem~\ref{unique} with only slight changes.
We notice that the morphism $g^R$ where the images $g(a)$ are reversed in order, is obtained
from $g$ by conjugation and the permutation $(1\, 3)$. 
\end{remark}

\section{Critical exponent VS number of factors of length 3}
Theorem~\ref{Actual3} exhibits a Dean word with critical exponent $2$
that avoids 6 factors of length 3.
In this section, we complement this result by considering the trade-off between
the critical exponent of an infinite Dean word and the number of avoided factors of length 3.
Negative results are obtained by standard backtracking. Positive results
are proved by uniform morphisms that generate infinite Dean words with suitable properties
using the method described in~\cite{Ochem:2004}.

\begin{lemma}
Let $w$ be a $\frac53$-free Dean word, then $|w|<62$.
\end{lemma}

\begin{lemma}
There exists an infinite $\frac53^+$-free Dean word avoiding 323.
\end{lemma}
The image of any $\frac{7}{5}^+$-free 4-ary word by the following 136-uniform morphism is a $\frac53^+$-free Dean word avoiding 323.

$$\begin{array}{ll}
 \texttt{0}\to&\texttt{0123032123012103012321012303210301210321230103}\\
              &\texttt{2101232103012303212301032101230321030123210123}\\
              &\texttt{01032123032101232103012103212301032101232103}\\
 \texttt{1}\to&\texttt{0123032123012103012321012303210301210321230103}\\
              &\texttt{2101230321030123210123010321230321012321030123}\\
              &\texttt{03212301032101230321030121032123032101232103}\\
 \texttt{2}\to&\texttt{0123032123012103012321012301032123012103012303}\\
              &\texttt{2123010321012321030121032123010321012303210301}\\
              &\texttt{23210123010321230321012321030121032123010321}\\
 \texttt{3}\to&\texttt{0123032123012103012321012301032123012103012303}\\
              &\texttt{2123010321012303210301232101230103212303210123}\\
              &\texttt{21030123032123010321012321030121032123010321}
\end{array}$$

\begin{lemma}
Let $w$ be a $\frac{17}{10}$-free Dean word and $d_3(w)\ge2$, then $|w|<289$.
\end{lemma}

\begin{lemma}
There exists an infinite $\frac{17}{10}^+$-free Dean word avoiding 030 and 232.
\end{lemma}
The image of any $\frac{7}{5}^+$-free 4-ary word by the following 358-uniform morphism is a $\frac{17}{10}^+$-free Dean word avoiding 030 and 232.


$$\begin{array}{ll}
 \texttt{0}\to&\texttt{0103230121032123010321012303212301210323010321230}\\
              &\texttt{3210123010321230121032301032101230321230121032123}\\
              &\texttt{0321012301032301210321230103210123032123010323012}\\
              &\texttt{1032123032101230103212301210323010321230321012301}\\
              &\texttt{0323012103212301032101230321230121032123032101230}\\
              &\texttt{1032123012103230103210123010323012103212301032101}\\
              &\texttt{2303212301032301210321230321012301032123012103230}\\
              &\texttt{103210123032123}\\
 \texttt{1}\to&\texttt{0103230121032123010321012303212301210323010321230}\\
              &\texttt{3210123010321230121032301032101230321230121032123}\\
              &\texttt{0321012301032301210321230103210123032123010323012}\\
              &\texttt{1032123032101230103212301210323010321012301032301}\\
              &\texttt{2103212301032101230321230121032123032101230103212}\\
              &\texttt{3012103230103212303210123010323012103212301032101}\\
              &\texttt{2303212301032301210321230321012301032123012103230}\\
              &\texttt{103210123032123}\\
 \texttt{2}\to&\texttt{0103230121032123010321012303212301210323010321230}\\
              &\texttt{3210123010321230121032301032101230103230121032123}\\
              &\texttt{0103210123032123010323012103212303210123010321230}\\
              &\texttt{1210323010321230321012301032301210321230103210123}\\
              &\texttt{0321230121032301032101230103230121032123032101230}\\
              &\texttt{1032123012103230103210123032123010323012103212301}\\
              &\texttt{0321012303212301210321230321012301032123012103230}\\
              &\texttt{103212303210123}\\
 \texttt{3}\to&\texttt{0103230121032123010321012303212301210323010321230}\\
              &\texttt{3210123010321230121032301032101230103230121032123}\\
              &\texttt{0103210123032123010323012103212303210123010321230}\\
              &\texttt{1210323010321230321012301032301210321230103210123}\\
              &\texttt{0321230121032123032101230103212301210323010321012}\\
              &\texttt{3032123010323012103212301032101230321230121032301}\\
              &\texttt{0321012301032301210321230321012301032123012103230}\\
              &\texttt{103212303210123}\\
\end{array}$$


\begin{lemma}
Let $w$ be a $\frac74$-free Dean word and $d_3(w)\ge3$, then $|w|<68$.
\end{lemma}

\begin{lemma}
There exists an infinite $\frac{7}{4}^+$-free Dean word avoiding 232, 212, 303, and 323.
\end{lemma}
The image of any $\frac{7}{5}^+$-free 4-ary word by the following 46-uniform  is a $\frac74^+$-free Dean word avoiding 232, 212, 303, and 323.


$$\begin{array}{l}
 \texttt{0}\to\texttt{0103210123012103010321030121032101230121030123}\\
 \texttt{1}\to\texttt{0103012101230121032101230103012103210121030123}\\
 \texttt{2}\to\texttt{0103012101230121030103210123012103210121030123}\\
 \texttt{3}\to\texttt{0103012101230103210301210321012103010321030123}
\end{array}$$

\begin{lemma}
Let $w$ be a $\frac{15}8$-free Dean word and $d_3(w)\ge5$, then $|w|<136$.
\end{lemma}

\begin{lemma}
There exists an infinite $\frac{15}8^+$-free Dean word avoiding 010, 032, 212, 303, and 323.
\end{lemma}
The image of any $\frac{7}{4}^+$-free ternary word by the following 100-uniform morphism is a $\frac{15}8^+$-free Dean word avoiding 010, 032, 212, 303, and 323.


$$\begin{array}{ll}
 \texttt{0}\to&\texttt{01232101230121030123210121030121012321030123210123}\\
              &\texttt{01210123210301210123012103012321030121012321012103}\\
 \texttt{1}\to&\texttt{01210123012103012321012301210123210121030121012321}\\
              &\texttt{03012321012103012101230121030123210123012101232103}\\
 \texttt{2}\to&\texttt{01210123012103012321012103012101232103012321012301}\\
              &\texttt{21030123210301210123210121030123210123012101232103}
\end{array}$$
\section{Critical exponent VS directedness}
A word $u$ is \emph{$d$-directed} if for every factor $f$ of $u$ of length $d$, the reverse of $f$, denoted by $f^R$, is not a factor of~$u$.
In this section, we consider the trade-off between the critical exponent of an infinite
Dean word and the smallest $d$ such that it is $d$-directed.
We use the same techniques as in the previous section for positive and negative results.
In order to verify that a word is $d$-directed, we only have to check for occurrences of factors of length $d$ and their reverse.

\begin{lemma}
There exists an infinite $\frac53^+$-free 12-directed Dean word.
\end{lemma}
The image of any $\frac{7}{5}^+$-free 4-ary word by the following 72-uniform morphism is a $\frac53^+$-free Dean word that is 12-directed.

$$\begin{array}{ll}
 \texttt{0}\to&\texttt{010321012303210301210323010321230321}\\
              &\texttt{030123210323012103212303210123210323}\\
 \texttt{1}\to&\texttt{010321012303210301210323010321230321}\\
              &\texttt{012321032301210321230321030123210323}\\
 \texttt{2}\to&\texttt{010321012303210301210323010321012321}\\
              &\texttt{032301210321230321012321030121032123}\\
 \texttt{3}\to&\texttt{010321012303210301210323010321012321}\\
              &\texttt{030121032123032101232103230121032123}
\end{array}$$

\begin{lemma}
Let $w$ be a $\frac{27}{16}$-free 11-directed Dean word, then $|w|<129$.
\end{lemma}

\begin{lemma}
There exists an infinite $\frac{27}{16}^+$-free 10-directed Dean word.
\end{lemma}
The image of any $\frac{7}{5}^+$-free 4-ary word by the following 564-uniform morphism is a $\frac{27}{16}^+$-free Dean word that is 10-directed.

{\tiny
$$\begin{array}{ll}
 \texttt{0}\to&\texttt{01032101230321030123210123032123010323012103012321012303210301232103230121032123010321}\\
              &\texttt{01230321230103230121030123210123032123010321012303210301232103230121030123210123032123}\\
              &\texttt{01032301210321230103210123032103012321012303212301032301210301232103230121032123010321}\\
              &\texttt{01230321230103230121030123210123032123010321012303210301232103230121032123010323012103}\\
              &\texttt{01232101230321230103230121032123010321012303210301232103230121030123210123032123010323}\\
              &\texttt{01210301232103230121032123010321012303212301032301210301232101230321030123210323012103}\\
              &\texttt{212301032301210301232101230321230103230121032123}\\
 \texttt{1}\to&\texttt{01032101230321030123210123032123010323012103012321012303210301232103230121032123010321}\\
              &\texttt{01230321230103230121030123210123032123010321012303210301232103230121030123210123032123}\\
              &\texttt{01032301210321230103210123032103012321012303212301032301210301232103230121032123010321}\\
              &\texttt{01230321230103230121030123210123032103012321032301210321230103230121030123210123032123}\\
              &\texttt{01032301210321230103210123032103012321032301210301232101230321230103230121030123210323}\\
              &\texttt{01210321230103210123032123010323012103012321012303212301032101230321030123210323012103}\\
              &\texttt{212301032301210301232101230321230103230121032123}\\
 \texttt{2}\to&\texttt{01032101230321030123210123032123010323012103012321012303210301232103230121032123010321}\\
              &\texttt{01230321230103230121030123210123032123010321012303210301232103230121030123210123032123}\\
              &\texttt{01032301210321230103210123032103012321012303212301032301210301232101230321030123210323}\\
              &\texttt{01210321230103230121030123210123032123010323012103212301032101230321030123210323012103}\\
              &\texttt{01232101230321230103230121030123210323012103212301032101230321230103230121030123210123}\\
              &\texttt{03210301232103230121032123010323012103012321012303212301032101230321030123210323012103}\\
              &\texttt{012321012303212301032301210301232103230121032123}\\
 \texttt{3}\to&\texttt{01032101230321030123210123032123010323012103012321012303210301232103230121032123010321}\\
              &\texttt{01230321230103230121030123210123032123010321012303210301232103230121030123210123032123}\\
              &\texttt{01032301210321230103210123032103012321012303212301032301210301232101230321030123210323}\\
              &\texttt{01210321230103230121030123210123032123010321012303210301232103230121030123210123032123}\\
              &\texttt{01032301210301232103230121032123010321012303212301032301210301232101230321030123210323}\\
              &\texttt{01210321230103230121030123210123032123010323012103212301032101230321030123210323012103}\\
              &\texttt{012321012303212301032301210301232103230121032123}
\end{array}$$
}

\begin{lemma}
Let $w$ be a $\frac74$-free 9-directed Dean word, then $|w|<114$.
\end{lemma}

\begin{lemma}
There exists an infinite $\frac74^+$-free 6-directed Dean word.
\end{lemma}
The image of any $\frac{7}{5}^+$-free 4-ary word by the following 40-uniform morphism is a $\frac74^+$-free Dean word that is 6-directed.

$$\begin{array}{ll}
 \texttt{0}\to&\texttt{0103012101232123012101230103012321230323}\\
 \texttt{1}\to&\texttt{0103012101232123010301230323012321230323}\\
 \texttt{2}\to&\texttt{0103012101230323012321230121012321230323}\\
 \texttt{3}\to&\texttt{0103012101230103012303230121012321230323}
\end{array}$$

\begin{lemma}
Let $w$ be a 5-directed Dean word, then $|w|<10$.
\end{lemma}

\section{Letter frequencies}
\begin{theorem}\label{minfreq}
 The minimum frequency of a letter in an infinite Dean word is $\frac{8}{59}$.
\end{theorem}

The image of every Dean word by the morphism $f$ below is a Dean word such that the frequency of the letter \texttt{3} is $\frac{8}{59}$.

$$\begin{array}{ll}
 f(\texttt{0})=&\texttt{010}\\
 f(\texttt{1})=&\texttt{30121012321012103012101230121032101210}\\
               &\texttt{30121012321012103210123012101232101210}\\
               &\texttt{301210123012103210121030121012321012103}\\
 f(\texttt{2})=&\texttt{212}\\
 f(\texttt{3})=&\texttt{32101210301210123210121032101230121012}\\
               &\texttt{32101210301210123012103210121030121012}\\
               &\texttt{321012103210123012101232101210301210123}
\end{array}$$

It is easy to see that the $f$-image of a Dean word is reduced.
Let us check that it is also square free.
A computer check rules out squares of period at most $500$.
Notice that $|f(\texttt{0})|=|f(\texttt{2})|=3$ and $|f(\texttt{1})|=|f(\texttt{3})|=115$.
Moreover, the factor $\texttt{010}$ (resp. $\texttt{212}$) only appears as the $f$-image of \texttt{0} (resp. \texttt{2}).
So the period of a potential square in the $f$-image of a Dean word must be a multiple of $|f(\texttt{01})|=118$.
Since the longest common prefix (resp. suffix) of $f(\texttt{1})$ and $f(\texttt{3})$ has length $1$, our square implies the existence of a square
with the same period and that contains the $f$-image of a letter as a prefix.
This forces a square in the pre-image by $f$, which is a contradiction.

A computer check shows that every Dean word of length $118$ and containing only $15$ occurrences of the letter \texttt{3} is not extendable. Thus $\frac{8}{59}$ is an optimal bound.

\section{Growth rate}

\newcommand{\minsuff}{\Lambda}
Let $T_n$ be the set of Dean words of length $n$. We use the same technique as in \cite{RC} (which is really close to the technique introduced in \cite{shurgrowthrate} that was itself inspired by \cite{Kolpakov2007}) to show the following.

\begin{theorem}\label{nbDeanWords}
  For all $n$, 
  $$|T_{n}|\ge 1.4581846^n\,.$$
\end{theorem}

Let $p=36$ and let $\mathcal{F}^{\le p}$ be the set of reduced words that contain no squares of period at most $p$. Let $\minsuff$ be the set of reduced words that are prefixes of minimal squares of period at most $p$.
For any $w$, we let $\minsuff(w)$ be the longest word from $\minsuff$ that is a suffix of $w$.
For any set of words $S$, and any $w\in \minsuff$, we let $S^{(w)}$ be the set of words from $S$ whose longest prefix that belongs to $\minsuff$ is $w$, that is $S^{(w)}=\{u\in S: \minsuff(u)=w\}$.
\begin{lemma}\label{bycomputer}
There exist coefficients $(C_w)_{w_\in\minsuff}$ such that $C_0>0$ and for all $v\in \minsuff$,
\begin{equation}\label{eqpassage}
\alpha C_v\le \sum_{\substack{a\in\{0,1,2,3\}\\va\in\mathcal{F}^{\le p}}} C_{\minsuff(va)}
\end{equation}
where $\alpha = 33075185 / 22682414 \approx 1.4581862847578$.
\end{lemma}
The proof of this lemma relies on computer verifications.
If one let $M\in\mathbb{Z}^{|\minsuff|\times|\minsuff|}$ be the matrix indexed over $\minsuff$ such that for all $u,v\in\minsuff$, $$M_{u,v}=|\{a\in\{0,1,2,3\}: u=\minsuff(va)\}|\,.$$
Then one can choose for the coefficients $\alpha$ the largest eigenvalue and for $(C_w)_{w\in\minsuff}$ any corresponding eigenvector. To find a vector close enough to this eigenvector, one can simply iterate the matrix (starting with the vector containing 1 everywhere) and renormalize the vector. One can then simply verify that the vector obtained after $n$ iterations (say $n=100$), has the desired property. We implemented this procedure in C++ \footnote{The program can be dowloaded at \url{https://www.lirmm.fr/~mrosenfeld/codes/Finding_the_coefficient_for_Dean_words.cpp}}
and it verifies Lemma \ref{bycomputer} in less than 10 minutes and using 9GB of RAM on a laptop.

Notice, that this procedure to find the largest eigenvalue of a matrix converges quickly, so we know that $\alpha$ should be a really good approximation of this quantity. 
Since this matrix is the adjacency matrix of an automaton that recognizes reduced words without squares of period at most $p$, we deduce that there exists a constant $C$ such that the number of such words of length $n$ is as most $C\cdot 1.4581863^n$. The same bound holds for $|T_n|$, that is, for all $n$, $|T_n|< C\cdot 1.4581863^n$. We deduce that Theorem \ref{nbDeanWords} is almost sharp.

For the rest of this section, let us fix coefficients $(C_w)_{w_\in\minsuff}$ that respect the conditions given by Lemma \ref{bycomputer}.
For each set $S$ of words, we let $$\widehat{S} =\sum_{w\in \minsuff} C_w |S^{(w)}|\,.$$
Whenever we mention the weight of a word $w$ in informal definitions,
we mean $C_{\minsuff(w)}$. We are now ready to state our main Lemma.

\begin{lemma}\label{mainlemma}
Let $\beta>1$ be a real number such that
$$\alpha-\frac{\beta^{2-\ceil{\frac{p+1}{2}}}}{\beta^2-1}\ge \beta\,.$$
  Then for all $n$, 
  $$\widehat{T_{n+1}}\ge \beta\widehat{T_n}\,.$$
\end{lemma}
\begin{proof}
  We proceed by induction on $n$. Let $n$ be an integer such that the lemma holds for any integer smaller than $n$ and let us show that 
  $\widehat{T_{n+1}}\ge \beta\widehat{T_n}$.
  
  By induction hypothesis, for all $i$,
\begin{equation}\label{IHPS}
 \widehat{T_{n}}\ge \beta^i\widehat{T_{n-i}}
\end{equation}

A word of length $n+1$ is \emph{good}, if its prefix of length $n$ is in $T_{n}$, if it is a reduced word and if it contains no square of period at most $p$. 
The set of good words is $G$. A word is \emph{wrong}, if it is good, but contains a square of period larger than $p$.
The set of wrong words is $F$.
Then for any $w$, $T_{n+1}=G\setminus F$ and

\begin{equation}\label{eqSGF}
 \widehat{T_{n+1}}\ge \widehat{G}-\widehat{F}
\end{equation}

Let us first lower-bound $\widehat{G}=\sum_{w\in \minsuff} |G(w)|C_w $.

The \emph{extensions} of any word $v\in S_n$ that belongs to $G$ are the words of the form 
$va$ where $a\in\{0,1,2,3\}$ and such that $va\in\mathcal{F}^{\le p}$.
By definition, $\minsuff(v)$ is the longest suffix of $v$ amongst prefixes of squares of period of length at most $p$. This implies that for any Dean word $v$ and for any letter $u$, $vu\in\mathcal{F}^{\le p}$ if and only if $\minsuff(v)u\in\mathcal{F}^{\le p}$. For the same reason, for any square-free word $v$ and for any word $u$, $\minsuff(vu)=\minsuff(\minsuff(v)u)$.
We then deduce that the contribution of the extentions of any word $v\in T_n$ to $\widehat{G}$ is 
$$\sum\limits_{\substack{a\in\{0,1,2,3\}\\va\in\mathcal{F}^{\le p}}} C_{\minsuff(va)}=
\sum\limits_{\substack{a\in\{0,1,2,3\}\\\minsuff(v)a\in\mathcal{F}^{\le p}}} C_{\minsuff(\minsuff(v)a)}
\ge\sum\limits_{\substack{a\in\{0,1,2,3\}\\\minsuff(v)a\in\mathcal{F}^{\le p}}} C_{\minsuff(\minsuff(v)a)}\,.$$
By Lemma \ref{bycomputer}, we deduce that the contribution of the extentions of any word $v\in T_n$ to $\widehat{G}$ is at least $\alpha C_{\minsuff(v)}$.
We sum the contribution over $T_n$ 
\begin{equation}
\widehat{G}\ge\sum_{v\in T_n}\alpha C_{\minsuff(v)}
=\sum_{u\in \minsuff}\alpha C_u |T_n^{(u)}|=\alpha \widehat{T_{n}}\label{boundG}
\end{equation}

Let us now bound $F$.
For all $i$, let $F_{i}$ be the set of words from $F$ that end with a square of period $i$.
Clearly, $F= \cup_{i\ge1} F_{i}$ and 
\begin{equation}\label{FFi}
  \widehat{F}\le \sum_{i\ge1} \widehat{F_{i}}\,.
\end{equation}

By definition of $G$ and $F$, for every $i\le p$, $ |F_i|=0$ and $\widehat{F_i}=0$. Moreover, since reduced words contain no square of odd period, for all $i$, $|F_{2i+1}|=0$ and $\widehat{F_{2i+1}}=0$.

Let us now upper-bound $\widehat{F_i}$ for any even $i>p$. 

Let $u\in F_{i}$ be a word. 
For the sake of contradiction suppose, $i\le\minsuff(u)$ and let $v$ be the square of period $i$ at the end of $u$ and let $k= |\minsuff(u)|$. 
By hypothesis, $v_1\cdots v_{i}= v_{i+1}\cdots v_{2i}$.
There exists $j\le p$ such that $\minsuff(u)$ has period $j$, but since $\minsuff(u)$ is the suffix of length $k$ of $v$, 
$v_{2i+1-k}\cdots v_{2i-j}= v_{2i+1+j-k}\cdots v_{2i}$ (using that $j\le p<i\le k$ one easily verifies that the indices are valid).
So in particular, using the two previous equations together, 
$$v_{j+1}\cdots v_{i} = v_{i+j+1}\cdots v_{2i}= v_{i+1}\cdots v_{2i-j}$$
(one easily verifies that all the indices are valid).
So there is a square inside $u$ which is not a suffix of $u$ which is a contradiction, since by definition of $F$ the only squares inside $u$ are suffixes of $u$.
Hence, $i>\minsuff(u)$.

For any $u\in F_i$, $u$ ends with a square of period $i$ so the last $i$ letters are uniquely determined by the prefix $v$ of length $|u|-i$. By definition, 
$v\in T_{n+1-i}$. Moreover, by the previous paragraph the suffix of size $|\minsuff(u)|+1$ of $u$ and $v$ are the same which implies $\minsuff(v) = \minsuff(u)$ and $v\in T_{n+1-i}^{(\minsuff(u))}$. So for every $w\in \minsuff$, any word of $ F_i^{(w)}$ is uniquely determined by a word of  $T_{n+1-i}^{(w)}$, which implies $|F_i^{(w)}|\le|T_{n+1-i}^{(w)}|$. By summing over all $w\in\minsuff$,
$$\widehat{F_i}\le\widehat{T_{n+1-i}}\,.$$
Together with \eqref{IHPS} it yields,
$$\widehat{F_i}\le\widehat{T_{n}}\beta^{1-i}\,.$$

We can now sum the $F_i$ to upper bound $F$ using equation \eqref{FFi},
$$\widehat{F}\le\sum_{i\ge\ceil{\frac{p+1}{2}}} \widehat{T_n} \beta^{1-i}=\frac{\beta^{2-\ceil{\frac{p+1}{2}}}}{\beta^2-1} \widehat{T_n}\,.$$

Using this bound and \eqref{boundG} with \eqref{eqSGF} yields
$$\widehat{T_{n+1}}\ge\widehat{T_{n}}\left(\alpha-\frac{\beta^{2-\ceil{ \frac{p+1}{2}}}}{\beta^2-1} \right)\,.$$
By theorem hypothesis we deduce 
$$\widehat{T_{n+1}}\ge\beta\widehat{T_{n}}$$ which concludes our proof.
\end{proof}

One easily verifies that $\beta=1.4581846$ satisfies the condtions of Lemma \ref{mainlemma} and we deduce the following corollary.
\begin{lemma}
  For all $n$, 
  $$\widehat{T_{n+1}}\ge 1.4581846\widehat{T_n}\,.$$
\end{lemma}
It implies that for all $n$, $\widehat{T_{n}}\ge 1.4581846^{n-1}\widehat{T_1}\ge 1.4581846^{n-1} C_0\,.$
There exists a constant $C$, such that $|T_{n}| \ge C 1.4581846^n$.
Using the fact that the set of Dean words is factorial, it is routine to deduce that for all $n$, 
  $$|T_{n}|\ge 1.4581846^n\,.$$ This proves Theorem \ref{nbDeanWords}.

Let us mention that, using the same technique, we were able to show that there exist at least $1.12^n$ $\frac53^+$-free Dean words of length $n$.

\section{Problems}


There are Dean words of length $n\ge7$ that do not occur in any infinite Dean words,
i.e., they cannot be extended in any infinite Dean word (but a finite amount). One example of such a word is $0103010$.

Due to minimality of $\Ex_l(w)$, there is no obvious relation between $\ex_l(w)$ and $\ex_{l+1}(w)$.
\begin{example}
We have $\Ex_3(g(\Hall))=S_1$ for the morphism $g$ of Theorem~\ref{Actual3}.
However, $\ex_4(g(\Hall))= \ex_5(g(\Hall))= \ex_6(g(\Hall))=0$.
\qed\end{example}

\begin{problem}\label{prob:1}
Are there infinite Dean words $w$ such that $\ex_l(w) >0$ for all $l \ge 3$?
\end{problem}

If Problem~\ref{prob:1} has a solution $w$ then the (nonempty) sets $\Ex_l(w)$
stabilise, i.e., for each $l$, there is a bound $N_l$ such that for any prefix $v$ of $w$ with $|v| \ge N_l$ one has
$\Ex_l(v)=\Ex_l(w)$ (since always $\Ex_l(w) \subseteq \Ex_l(v))$. We conjecture that in such a case the integers $\ex_l(w)$ are very small.

As mentioned, there are Dean words that are non extendable.
We formally define the set of extendable Dean worrds as following:
\[
\Omega = \{ w \mid \text{ for all $n$, there is a Dean word } u_n w v_n \text{ with } |u_n|, |v_n| \ge n\}.
\]

\begin{problem}
Give a characterisation of the set $\Omega$.
\end{problem}

For a Dean word $w$, let $m_l(w) = |M_l(w)|$ denote the size of the set of \emph{minimal extendable avoided words}:
\[
M_l(w) = \Omega \cap \Ex_l(w).
\]

\begin{problem}\label{prob:2}
Are there infinite Dean words $w$ such that $m_l(w)=0$ for all $l$?
\end{problem}
A solution word for Problem~\ref{prob:2} would contain all words in $\Omega$. 
Clearly, if $m_j(w)=0$ then also $m_k(w)=0$ for all $k < j$.

\begin{example}
The Dean word $w=010301210123032123210323010301$ of length 30 satisfies
$\ex_2(w)=\ex_3(w)=\ex_4(w)=0$ but $\ex_5(w)=16$.
\qed\end{example}

\begin{problem}
Each Dean word $w \in \Al_{3}^*$ alternates between even and odd letters,
and thus $w$ is a letter-to-letter shuffle of even
$w_{e} \in \{ 0,2\}^*$ and odd words $w_{o} \in \{1,3\}^*$.
What are the \emph{forbidden words} $v \in \{ 0,2\}^*$ that are not of the form $v = w_{e}$?
\end{problem}

\section*{Appendix}

\noindent
(1)
Below we list in the cycle form the permutations of $\Al_4$ that preserve reduced words (apart from the identity):
\[
(0\, 2),\ 
(1\, 3),\
(0\, 2)(1\, 3),\
(0\, 1)(2\, 3), \ 
(0\, 3)(2\, 1), \
(0\, 1 \, 2 \, 3),\
(0\, 3 \, 2 \, 1).\
\]

\medskip
\noindent
(2)
The list of the factors of length~7 of the Hall-Thue word $\Hall$ needed in Theorem~\ref{Actual3}: 
\begin{align*}
0120210&&
1202101&&
2021012&&
0210121&&
2101210\\
1012102&&
0121020&&
1210201&&
2102012&&
1020120\\
0201202&&
2012021&&
1202102&&
2021020&&
0210201\\
1020121&&
0201210&&
2012101&&
0121012&&
1210120\\
2101202&&
1012021&&
\end{align*}

\bibliographystyle{plain}
\bibliography{bib-small}

\end{document}